\documentclass[oneside,english]{amsart}
\usepackage[T1]{fontenc}
\usepackage[utf8]{inputenc}
\usepackage{verbatim}
\usepackage{amstext}
\usepackage{amsthm}
\usepackage{amssymb}

\makeatletter
\numberwithin{equation}{section}
\numberwithin{figure}{section}
\theoremstyle{plain}
\newtheorem{thm}{\protect\theoremname}
\theoremstyle{definition}
\newtheorem{defn}[thm]{\protect\definitionname}
\theoremstyle{plain}
\newtheorem{lem}[thm]{\protect\lemmaname}
\theoremstyle{plain}
\newtheorem{cor}[thm]{\protect\corollaryname}
\theoremstyle{remark}
\newtheorem{rem}[thm]{\protect\remarkname}

\usepackage{hyperref}
\usepackage[T2A]{fontenc}

\usepackage{tikz}
\usetikzlibrary{calc,math,intersections}
\usetikzlibrary{shapes.geometric}
\usetikzlibrary{patterns}
\usetikzlibrary{arrows.meta}
\mathcode`l="8000
\begingroup
\makeatletter
\lccode`\~=`\l
\DeclareMathSymbol{\lsb@l}{\mathalpha}{letters}{`l}
\lowercase{\gdef~{\ifnum\the\mathgroup=\m@ne \ell \else \lsb@l \fi}}%
\endgroup

\AtBeginDocument{
  
}

\makeatother

\usepackage{babel}
\providecommand{\corollaryname}{Corollary}
\providecommand{\definitionname}{Definition}
\providecommand{\lemmaname}{Lemma}
\providecommand{\remarkname}{Remark}
\providecommand{\theoremname}{Theorem}

\begin{document}
\global\long\def\conv{\operatorname{conv}}%

\global\long\def\diag{\operatorname{diag}}%

\global\long\def\diam{\operatorname{diam}}%

\global\long\def\dist{\operatorname{dist}}%

\global\long\def\rank{\operatorname{rank}}%

\global\long\def\Span{\operatorname{span}}%

\global\long\def\id{\operatorname{id}}%

\global\long\def\Aff{\operatorname{Aff}}%

\global\long\def\relint{\operatorname{relint}}%

\global\long\def\Int{\operatorname{int}}%

\global\long\def\R{\mathbb{R}}%

\global\long\def\Simplexe{\mathcal{S}}%

\global\long\def\Vertices{\mathcal{\mathcal{V}}}%

\title[A Strengthened Alexandrov Maximum Principle]{A Strengthened Alexandrov Maximum Principle\\ or Uniform Hölder
Continuity\\ for Solutions of the Monge--Ampère Equation\\ with
Bounded Right-Hand Side}
\author{Lukas Gehring}
\address{Institut für Mathematik, Friedrich-Schiller-Universität Jena, Ernst-Abbe-Platz
2, 07743 Jena, Germany}
\email{lukas.gehring@uni-jena.de}
\date{\today}
\begin{abstract}
This article is about the convex solution $u$ of the Monge--Ampère
equation on an at least 2-dimensional open bounded convex domain with
Dirichlet boundary data and nonnegative bounded right-hand side.

For convex functions with zero boundary data, an Alexandrov maximum
principle $\left|u(x)\right|\le C\dist(x,\partial\Omega)^{\alpha}$
is equivalent to (uniform) Hölder continuity with the same constant
and exponent. Convex $\alpha$-Hölder continuous functions are $W^{1,p}$
for $p<1/(1{-}\alpha)$. We prove Hölder continuity with the exponent
$\alpha=2/n$ for $n\ge3$ and any $\alpha\in(0,1)$ for $n=2$, provided
that the boundary data satisfy this Hölder continuity, and show that
these bounds for the exponent are sharp. The only means is to bound
the Hessian determinant of a certain explicit function on an $n$-dimensional
cylinder and to use the comparison princple. 
\end{abstract}

\thanks{This project received inspiration from the Almighty from whom all
ideas issue and funding from the European Union's Horizon 2020 research
and innovation programme (Grant agreement No.~891734).}
\maketitle
\begin{flushright}
\emph{}%
\par\end{flushright}

\section{Introduction}

The subject of this article is the convex solution of the Monge--Ampère
equation
\begin{equation}
\begin{cases}
\det D^{2}u & =f\quad\text{in }\Omega,\\
u & =g\quad\text{on }\partial\Omega,
\end{cases}\label{eq:DP}
\end{equation}
for an open, bounded, and convex set $\Omega\subset\mathbb{R}^{n}$,
$n\ge2$, $0\le f\in L^{1}(\Omega)$ and a convex function $g\in C(\overline{\Omega})$.
The suitable notion of a weak formulation for this equation, the \emph{Alexandrov
formulation \cite{alexandrov}, }uses the subdifferential

\[
\partial u(x):=\left\{ \xi\in\mathbb{R}^{n}~\middle|~u\ge u(x)+\left<\xi,\bullet-x\right>\text{ in }\Omega\right\} ,
\]
the Lebesgue measure $\left|\bullet\right|$, and the Monge--Ampère
measure
\[
\mu_{u}:A\mapsto\left|\bigcup_{x\in A}\partial u(x)\right|
\]
and seeks a convex function $u\in C(\overline{\Omega})$ with 
\begin{equation}
\begin{cases}
\mu_{u} & =fdx\quad\text{in }\Omega,\\
u & =g\quad\text{on }\partial\Omega.
\end{cases}\label{eq:Alexandrov formulation}
\end{equation}

\subsection{Related prior work}

When this article was already completed, the author got to know that
the main result (\ref{lem:new_Alexandrov}) was already proven by
Caffarelli in \cite{caffarelli_localization} for $n\ge3$. The Alexandrov
formulation provides uniqueness of the convex solution for (\ref{eq:Alexandrov formulation})
and existence if $g=0$ or $\Omega$ is strictly convex (even for
finite Borel measures $\nu$ instead of $fdx$) \cite[Theorem~2.13, p.~20, Theorem~2.14, p.~24]{Figalli2017}.
A brief history of the Monge--Ampère equation can be found in the
introduction of \cite{Figalli2017}. The regularity of solutions has
been subject of intensive research. First of all, Alexandrov showed
his maximum principle \cite{alexandrov} for $u$ vanishing on the
boundary, which implies that $u\in C^{0,1/n}$ (see below). Trudinger--Urbas
\cite{trudinger_urbas} showed that a uniformly convex $C^{1,1}$
domain, a right-hand side $f(x,u,\nabla u)$ bounded from above by
$\mu(|u|)\dist(x,\partial\Omega)^{\beta}(1{+}|\nabla u|^{2})^{\alpha}$
with a positive and nodecreasing $\mu$, $\alpha\ge0$, $\beta\ge0$,
$\beta\ge\alpha{-}n{-}1$ and further conditions lead to $u\in C^{2}\cap C^{0,1}$.
Urbas \cite{urbas} proved a global Hölder estimate for $C^{1,1}$
domains and a right-hand side $f(x,u,\nabla u)$ bounded from below
by $\mu\dist(x,\partial\Omega)^{\beta}(1+\left|\nabla u\right|^{2})^{\alpha}$.
In a later work \cite{urbas_gauss}, he dealt with the equation of
prescribed Gauss curvature. Furthermore, there are estimates at the
boundary, if $u$ behaves on the boundary like $|x|^{2}$ \cite{savin_localization,savin_pointwise,savin_w2p}.
And there are global $W^{1,p}$ and $W^{2,p}$ estimates for solutions
of uniformly elliptic equations \cite{winter_n_w1p_at_the_b} (the
Monge--Ampère equations are elliptic but not uniformly) and the linearized
Monge--Ampère equations \cite{le_w2p,le_global_w1p}. A global $W^{1,p}$
estimate for the Monge--Ampère equation was not known to the author
yet.

\subsection{Basic properties}

The Monge--Ampère measure has the following scaling properties: If
$\det D^{2}u=f$, more generally if $\mu_{u}=fdx$, then
\begin{equation}
\mu_{cu\circ A}=c^{n}\left(\det L_{A}\right)^{2}f\circ Adx\label{eq:scaling properties}
\end{equation}
for $c\ge0$, an affine map $A:\mathbb{R}^{n}\rightarrow\mathbb{R}^{n}$,
and its linear map $L_{A}$. As an elliptic equation, the Monge--Ampère
equation satisfies a \emph{comparison principle}: If for continuous
convex functions $u$ and $v$ on the closure of a bounded open set
$\mathcal{U}\subset\Omega$,
\begin{align}
 &  & u & \le v & \text{on }\partial\mathcal{U}\nonumber \\
 & \text{and} & \mu_{u} & \ge\mu_{v} & \text{in }\mathcal{U},\label{eq:comparision principle}\\
 & \text{then also} & u & \le v & \text{in }\mathcal{U}.\nonumber 
\end{align}
The Monge--Ampère measure is superadditive \cite[Lemma~2.9, p.~17]{Figalli2017},
i.e.\ 
\begin{equation}
\mu_{u+v}\ge\mu_{u}+\mu_{v}.\label{eq:superadditivity}
\end{equation}
The Alexandrov maximum principle (AMP) says that the solution $u$
of (\ref{eq:Alexandrov formulation}) with $g=0$ satisfies $\left|u(x)\right|\le C\allowbreak\dist(x,\partial\Omega)^{\alpha}$
for $\alpha=1/n$ \cite[Theorem 2.8]{alexandrov,Figalli2017}. 

\subsection{Outline}

For nonvanishing boundary values, the AMP is generalized by Hölder
continuity of a convex function (see (\ref{eq:H=0000F6lderstetigkeit AMP}))
in the preliminary Section \ref{sec:Preliminaries}, as mentioned
in \cite[Lemma~3.3]{trudinger_wang}. Additionally, an $\alpha$-Hölder
continuous convex function $u$ satisfies $\left|\nabla u(x)\right|\lesssim\dist(x,\partial\Omega)^{\alpha-1}$
for every choice of $\nabla u(x)\in\partial u(x)$ (see Lemma \ref{lem:gradient differenzenquotient})
which leads to $u\in W^{1,p}$ for $p<1/(1{-}\alpha)$ (see Lemma
\ref{lem:H=0000F6lder Sobolev}) aside from other preliminaries. In
Section \ref{sec:certain explicit function}, the Hessian determinants
of certain explicit functions of the form $w(x)=a(x_{1})b(x')$ on
a cylinder $[0,h]\times\overline{B}_{\rho}^{n-1}$ are bounded from
\emph{below}. This function has the aimed growth behavior $\left|w(x)\right|\leq x_{1}^{\alpha}$
(for $n=2$ $\left|w(x)\right|\lesssim x_{1}(1{-}\ln x_{1})$, respectively)
near the boundary. Then, in Section \ref{sec:results}, the comparison
principle enables us to infer from that the main result Theorem \ref{thm:H=0000F6lderstetigkeit}
that the solution of (\ref{eq:Alexandrov formulation}) with right-hand
side bounded from \emph{above} is Hölder continuous on an arbitrary
open bounded convex domain. This implies a bound for the $L^{p}$
norm of the gradient of $u$ in Corollary \ref{thm:w1p}. In Section
\ref{sec:converse bounds}, an \emph{upper bound }for the Hessian
determinant of the above explicit function shows that the bounds for
the exponents in the Hölder continuity and in the $W^{1,p}$ estimate
are sharp.

\section{\label{sec:Preliminaries}Preliminaries}

\subsection{Notation}

The expression $\left|A\right|$ means the absolute value of $A$,
if $A$ is a real number or a real valued function; the Euclidean
norm $\sqrt{1/n\sum_{i=1}^{n}A_{i}^{2}}$, if $A\in\mathbb{R}^{n}$
and the spectral norm, if $A$ is a matrix; and the Lebesgue measure
of $A$, if $A\subset\mathbb{R}^{n}$ is a Lebesgue set. Let $\delta_{ij}$
be the Kronecker delta, i.e., 1 if $i=j$ and 0 otherwise. $\overline{Z}$
denotes the closure of the set $Z$. Throughout this article, $\Omega\in\mathbb{R}^{n}$
is an open, bounded, and convex domain. In contrast, $\Omega_{h}$
denotes the closed set 
\begin{equation}
\Omega_{h}:=\left\{ x\in\Omega~\middle|~\dist(x,\partial\Omega)\ge h\right\} .\label{eq:Omega_h}
\end{equation}
Let $B_{R}^{k}\subset\mathbb{R}^{k}$ be the $k$-dimensional ball
of radius $R$ around the origin. For $x=(x_{1},\dots,x_{n})\allowbreak\in\mathbb{R}^{n}$,
let $x':=(x_{2},\dots,x_{n})$. The line segment between two points
$x$ and $y\in\mathbb{R}^{n}$ is denoted by $\overline{xy}.$ Whenever
we write that $u=g$ on $\partial\Omega$, we imply that $u$ extends
continuously to the boundary. The convex envelope of a function $u$
is denoted by $u_{*}$. $C$ (also with an index) denotes a real constant
which may vary from line to line. Let $\Aff(\mathbb{R}^{n})$ be the
set of affine bijections $\mathbb{R}^{n}\rightarrow\mathbb{R}^{n}$.
By an abuse of notation, we identify an affine map $L$ with its linear
map. The spectral norm $\left|L\right|$ and $\det L$ always refer
to the linear map. For $p\in[1,\infty)$, $\left\Vert u\right\Vert _{p}:=\left(\int_{\Omega}\left|u(x)\right|^{p}dx\right)^{1/p}$
denotes the $L^{p}$ norm of $u$ on $\Omega$ and for $p=\infty$,
it denotes the usual essential supremum.
\begin{defn}
\label{def:Cylinder}For $h,\rho>0$, we will use cylinders
\[
K_{h,\rho}:=(0,h)\times B_{\rho}^{n-1}.
\]
\end{defn}

\subsection{Uniformly continuous convex functions}
\begin{defn}
For a function $u:\overline{\Omega}\rightarrow\mathbb{R}^{n}$ let
\[
\omega_{u}(\delta):=\sup_{\substack{x,y\in\overline{\Omega}\\
\left|x{-}y\right|\le\delta
}
}\left|u(x){-}u(y)\right|
\]
be the \emph{modulus of continuity of $u$.}
\end{defn}

Obviously, $\omega_{u}$ is nonnegative, monotonically increasing,
sublinear in $u$: For $\lambda\ge0$ 
\begin{equation}
\omega_{\lambda u+v}\le\lambda\omega_{u}+\omega_{v},\label{eq:sublinearity}
\end{equation}
and we have the following composition rule: For $u:\overline{\Omega}_{1}\to\overline{\Omega}_{2}$
and $v:\overline{\Omega}_{2}\to\mathbb{R}^{n}$ it holds that $\omega_{v\circ u}\le\omega_{v}\circ\omega_{u}.$
For affine $u$, this reads 
\begin{equation}
\omega_{v\circ u}(\delta)\le\omega_{v}(\left|u\right|\delta).\label{eq:composition rule}
\end{equation}

Now let $\overline{\Omega}$ be compact and $u:\overline{\Omega}\to\mathbb{R}$
be convex. Since on a line segment $x+(0,\infty)z\cap\overline{\Omega}$,
the function $0<\lambda\mapsto\frac{u(x)-u(x+\lambda z)}{\lambda}$
is monotonically decreasing, 
\begin{equation}
\omega_{u}(\bullet)/\bullet\quad\text{is monotonically decreasing.}\label{eq:sublinearity of modulus}
\end{equation}
The setting is illustrated in Figure \ref{fig:H=0000F6lderbild}:
For any $x,y\in\overline{\Omega}$, the quantity $u(x)-u(y)$ will
not decrease if the line segment $\overline{xy}$ is translated in
the direction $x{-}y$ to the boundary of $\partial\Omega$, i.e.
if $x$ and $y$ are replaced by $x'=x+\lambda(x{-}y)\in\partial\Omega$
and $y'=y+\lambda(x{-}y)$, respectively, with $\lambda\geq0$. 
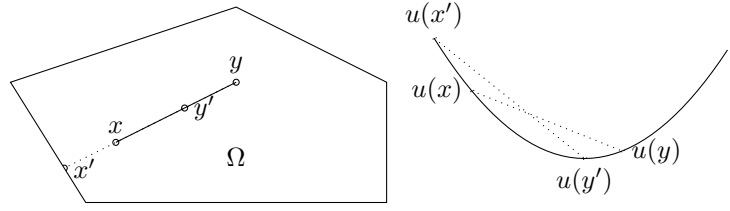
\begin{figure}
\begin{tikzpicture}[scale=2] 
\path[name path=rand] (0.5,1.5)--(1,.7); 
\draw[clip] (1,.7)--(3,.7)--(3,1.5)--(2,2)--(0.5,1.5)--cycle; 
\draw (2,1) node {$\Omega$}; the modulus of continuity
\coordinate (x) at (1.2,1.1); 
\coordinate (y) at (2,1.5); 
\draw[name path=sehne,dotted] (-1,0)--(y); 
\draw (x)node[above]{$x$}circle [radius=.2mm]--(y)node[above]{$y$}circle [radius=.2mm]; 
\draw[name intersections={of =rand and sehne, by={w}}] 
(w) node [right] {$x'$}circle [radius=.2mm] 
($(w)+(y)-(x)$) node [right] {$y'$} circle [radius=.2mm]; 
\end{tikzpicture} 
\begin{tikzpicture}[yscale=.4] 
\pgfplothandlerlineto 
\pgfplotfunction{\x}{-2,-1.9,...,2}{\pgfpointxy{\x}{\x*\x}} 
\pgfusepath{stroke} 
\draw[dotted] 
(-1.5,2.25) node [left]{$u(x)$} circle [radius=.3mm] -- 
(.5,.25) node [right]{$u(y)$} circle [radius=.3mm] 
(-2,4) node [above]{$u(x')$} circle [radius=.3mm] -- 
(0,0) node [below]{$u(y')$} circle [radius=.3mm]; 
\end{tikzpicture}\caption{\label{fig:H=0000F6lderbild}Translate the line segment to the boundary.
The difference will not decrease.}
\end{figure}
This shows
\begin{equation}
\omega_{u}(\delta)=\sup_{\substack{x\in\partial\Omega,y\in\overline{\Omega}\\
|x{-}y|\le\delta
}
}u(x)-u(y)\label{eq:H=0000F6lderstetigkeit boundary}
\end{equation}
for convex functions $u$ on compact domains $\overline{\Omega}$.
If $u$ \emph{vanishes} on the boundary additionally, the latter quantity
equals
\begin{equation}
\omega_{u}(\delta)=\sup_{\substack{y\in\overline{\Omega}\\
\dist(y,\partial\Omega)\le\delta
}
}-u(y)\label{eq:H=0000F6lderstetigkeit AMP}
\end{equation}
which is closely related to the AMP. Another direct consequence of
(\ref{eq:H=0000F6lderstetigkeit boundary}) is the following \emph{comparision
principle of the modulus of continuity}: For a convex function $u$
on $\overline{\Omega}$ with a (nonnecessarily convex) minorant $v$
with 
\begin{align}
u & \ge v\quad\text{in }\Omega,\nonumber \\
u & =v\quad\text{on }\partial\Omega,\label{eq:comparision4modulus}\\
\text{it holds that}\quad\omega_{u} & \leq\omega_{v}.\nonumber 
\end{align}
 In particular, this implies that for fixed boundary conditions $g|_{\partial\Omega}$
(with a convex $g\in C(\overline{\Omega})$), the modulus $\omega_{\tilde{g}}$
of any convex extension $\tilde{g}$ of these is minimized by their
convex envelope $\left(g|_{\partial\Omega}\right)_{*}$. The superadditivity,
the subadditivity and the comparision principles of the Monge--Ampère
measure and of the modulus compose to the following subadditivity
between the Monge--Ampère measure and the modulus: 
\begin{lem}
\label{lem:subadditivity MAM modul}For $i\in\{1,2\}$, let $u_{i}$
be convex functions on $\overline{\Omega}$. Let $u_{12}$ be convex
and continuous with 
\begin{align*}
\mu_{u_{12}} & \leq\mu_{u_{1}}+\mu_{u_{2}}\quad\text{in }\Omega,\\
u_{12} & =u_{1}+u_{2}\quad\text{on }\partial\Omega.
\end{align*}
Then 
\[
\omega_{u_{12}}\le\omega_{u_{1}}+\omega_{u_{2}}.
\]
\end{lem}

\begin{proof}
The superadditivity (\ref{eq:superadditivity}) says that $\mu_{u_{12}}\le\mu_{u_{1}+u_{2}}$,
the comparision principle (\ref{eq:comparision principle}) that $u_{12}\ge u_{1}+u_{2}$,
and the comparision principle (\ref{eq:comparision4modulus}) and
the subadditivity (\ref{eq:sublinearity}) of the modulus of continuity
show that
\[
\omega_{u_{12}}\le\omega_{u_{1}+u_{2}}\le\omega_{u_{1}}+\omega_{u_{2}}.
\]
\end{proof}
\begin{lem}[Bound for the gradient with $\omega_{u}$]
\label{lem:gradient differenzenquotient}Let $u:\overline{\Omega}\rightarrow\mathbb{R}$
be convex and $x\in\Omega$. Then for any vector $p\in\partial u(x)$
\[
\left|p\right|\le\sup_{y\in\partial\Omega}\frac{\left|u(y)-u(x)\right|}{\left|y{-}x\right|}\le\frac{\omega_{u}(\dist(x,\partial\Omega))}{\dist(x,\partial\Omega)}.
\]
\end{lem}

\begin{proof}
Let $y$ be a point on $\partial\Omega$ with $\lambda\left(y{-}x\right)=p$
for some $\lambda\geq0$. ($y$ is unique if $p\neq0$, because $\Omega$
is bounded and convex.) Then convexity of $u$ in $x$ yields
\begin{eqnarray*}
u(y)-u(x) & \ge & \langle p,y{-}x\rangle=\left|p\right|\left|y{-}x\right|.
\end{eqnarray*}
This shows $\left|p\right|\le\sup_{y\in\partial\Omega}\left|u(y)-u(x)\right|/\left|y{-}x\right|$,
which is, since $\omega_{u}(\bullet)/\bullet$ decreases monotonically
(\ref{eq:sublinearity of modulus}), bounded by $\omega_{u}(\dist(x,\partial\Omega))/\dist(x,\partial\Omega)$.
\end{proof}
As a composition of the scaling properties (\ref{eq:scaling properties}),
the composition rule (\ref{eq:composition rule}) of the modulus and
the chain rule of calculus, the next lemma sums up how to reduce solutions
for arbitrary bounded right-hand sides and domains to the normalized
case $\Lambda=1$ and $\Omega$ included in the unit ball $B_{1}^{n}$.
\begin{lem}[Reduction to a normalized problem]
\label{lem:Reduction to normalized}If $u\in C(\overline{\Omega})$
is convex, $\mu_{u}=fdx$, and $L\in\Aff(\mathbb{R}^{n})$, then
\[
v:=\Lambda^{1/n}\left|\det L\right|^{-2/n}u\circ L\in C(L\overline{\Omega})
\]
is convex, satisfies
\[
\mu_{v}=\Lambda f\circ Ldx,
\]
and 
\begin{align*}
\omega_{v}(\delta) & \le\Lambda^{1/n}\left|\det L\right|^{-2/n}\omega_{u}\left(\left|L\right|\delta\right),\\
\left|\nabla v\right| & \le\Lambda^{1/n}\left|L\right|\left|\det L\right|^{-2/n}\left|\nabla u\right|.
\end{align*}
\end{lem}

\subsection{Hölder continuous convex functions are Sobolev}

The next two propositions will be used to estimate the integral $\int_{\Omega}\dist(x,\partial\Omega)^{q}dx$.
\begin{lem}
\label{lem:Lipschitz} Recall that $\Omega_{h}:=\left\{ x\in\Omega~\middle|~\dist(x,\partial\Omega)\ge h\right\} $,
let 
\begin{align*}
F & :=\left\{ (x,|x{-}y|,y)~\middle|~x\in\partial\Omega,y\in\Omega,\left|x{-}y\right|=\dist(y,\partial\Omega)\right\} ,
\end{align*}
 and 
\[
\tilde{\Omega}:=\pi_{12}(F):=\left\{ (x,h)~\middle|~\exists y\in\Omega:(x,h,y)\in F\right\} .
\]
Then
\begin{enumerate}
\item $\Omega_{h}$ is convex.
\item $F$ is univalent with respect to $y$, i.e.\ given $x\in\partial\Omega$
and $h\in(0,\infty)$, there is at most one $y$ such that $(x,h,y)\in F$.
The corresponding function 
\begin{align*}
f:\tilde{\Omega} & \rightarrow\Omega\\
(x,h) & \mapsto y,\quad\text{where }(x,h,y)\in F
\end{align*}
equals the projection $f(x,h)=\pi(x,\Omega_{h}):=\arg\min_{y\in\Omega_{h}}\left|x{-}y\right|$.
(Nevertheless we want to retain the domain $\tilde{\Omega}$ of $f$.)
\item $f$ is surjective.
\item If $(x,\left|x{-}y\right|,y)\in F,$ then also $(x,\left|x{-}z\right|,z)\in F$
for all $z\in\overline{xy}$.
\item Equip $\partial\Omega\times\mathbb{R}$ with the metric induced by
the Euclidean norm on $\mathbb{R}^{n+1}$:
\[
d\left((x,a),(y,b)\right):=\left|(x,a)-(y,b)\right|=\sqrt{\left|x{-}y\right|^{2}+\left|a{-}b\right|^{2}}.
\]
Then $f$ is 1-Lipschitz.
\end{enumerate}
\end{lem}

\begin{proof}
If $\Omega$ is the empty set, then also $\Omega_{h}$, $F$ and $\partial\Omega$
are empty, $f$ is the function $\emptyset\rightarrow\emptyset$ and
the statements are trivial. Let us investigate the case where $\Omega\neq\emptyset$. 

(1) $x\in\Omega_{h}$, if and only if $x{+}B_{h}^{n}\subset\Omega$.
Let $x,y\in\Omega_{h}.$ Since $\Omega$ is convex, it contains the
convex hull $\overline{xy}+B_{h}$ of $\left(x{+}B_{h}^{n}\right)\cup\left(y{+}B_{h}^{n}\right)$,
hence $\lambda x+(1{-}\lambda)y+B_{h}^{n}\subset\Omega$ and $\lambda x+(1{-}\lambda)y\in\Omega_{h}$.

(2) If $(x,h,y)\in F$, then $y\in\Omega_{h}$ and there cannot be
any point in $\Omega_{h}$ closer to $x$ than $y$. This means that
$y$ is the projection of $x$ onto $\Omega_{h}$, which is unique
\cite[Theorem V.2, p.~79]{brezis}. 

(3) Since $\partial\Omega$ is compact, every $y\in\Omega$ has a
proximum $x=\arg\min_{x\in\partial\Omega}\left|x{-}y\right|$ on $\partial\Omega$
and $f(x,\left|x{-}y\right|)=y$.

(4) We have to show that $\dist(z,\partial\Omega)=|x{-}z|$. The triangle
inequality, the infimizing property of $\dist$ and the fact that
$z$ lies on the line segment $\overline{xy}$ show that
\[
\dist(y,\partial\Omega)\le\dist(z,\partial\Omega)+\left|z{-}y\right|\leq\left|x{-}z\right|+\left|z{-}y\right|=\left|x{-}y\right|.
\]
 Since $\dist(y,\partial\Omega)=\left|x{-}y\right|$, all terms of
this line must be equal and $\dist(z,\partial\Omega)=\left|x{-}z\right|$.

(5) Let $(x,a,z),(y,b,w)\in F$ and without loss of generality let
$a\le b$. According to (4), there is $v\in\overline{yw}$ with $(y,a,v)\in F$.
We decompose
\begin{align*}
f(x,a)-f(y,b) & =\left(f(x,a){-}f(y,a)\right)+\left(f(y,a){-}f(y,b)\right)\\
 & =\left(\pi(x,\Omega_{a}){-}\pi(y,\Omega_{a})\right)+\left(f(y,a){-}f(y,b)\right).
\end{align*}
Since $v\in\overline{yw}$, the length of the second term $\left|f(y,a)-f(y,b)\right|\allowbreak=b{-}a$.
Since $\Omega_{h}$ is closed and convex, the projection $\pi(\bullet,\Omega_{a})$
is 1-Lipschitz according to \cite[Proposition V.3, p.~80]{brezis},
so the length of the first term $\left|f(x,a)-f(y,a)\right|\allowbreak\le\left|x{-}y\right|$.
Because of the well-known fact that for $\Omega_{a}$ as for any convex
set,
\[
\left<y{-}\pi(y,\Omega_{a}),z{-}\pi(y,\Omega_{a})\right>\le0\quad\text{for all }z\in\Omega_{a}
\]
\cite[Theorem V.2, p.~79]{brezis}, and since $f(y,a)-f(y,b)$ points
in the same direction as $y-f(y,a)$, we know that the angle between
the vectors is not acute:
\[
\left<f(y,a){-}f(y,b),f(x,a){-}f(y,a)\right>\le0.
\]
This implies that 
\begin{align*}
\left|f(x,a)-f(y,b)\right| & \leq\sqrt{\left|x{-}y\right|^{2}+\left|a{-}b\right|^{2}}.
\end{align*}
\end{proof}
\begin{cor}[Volume of the layers]
\label{cor:Hausdorff}Let $\Omega\subset B_{R}^{n}$ be a bounded
convex set. Then for all $0<a<b<\infty$
\begin{equation}
\left|\Omega_{a}\setminus\Omega_{b}\right|\leq\mathcal{H}^{n-1}(\partial\Omega)(b{-}a)\leq\mathcal{H}^{n-1}\left(\partial B_{1}^{n}\right)R^{n-1}(b{-}a),
\end{equation}
where $\mathcal{H}^{n-1}(\partial\Omega)$ denotes the $(n{-}1)$-dimensional
Hausdorff measure.
\end{cor}

\begin{proof}
In terms of Lemma \ref{lem:Lipschitz}, the set of interest equals
\[
\Omega_{a}\setminus\Omega_{b}=f\left(\tilde{\Omega}\cap(\partial\Omega\times[a,b))\right).
\]
Since $f$ is 1-Lipschitz, $\left|\Omega_{a}\setminus\Omega_{b}\right|$
can be estimated from above by the $n$-dimensional Hausdorff measure
of the preimage $\tilde{\Omega}\cap(\partial\Omega\times[a,b))$ (cf.\ \cite[Theorem~2.8, p.~97]{Evans}),
which is at most
\[
(b{-}a)\mathcal{H}^{n-1}(\partial\Omega).
\]
The surface measure of $\Omega$ is bounded by the surface measure
of a ball including $\Omega$ \cite[Corollary A.16, p.~172]{stefani,Figalli2017},
which shows the second inequality.
\end{proof}
\begin{lem}[Hölder continuous convex functions are Sobolev]
 \label{lem:H=0000F6lder Sobolev}Let $\overline{\Omega}\subset B_{R}^{n}$,
$r:=\max\dist(\bullet,\partial\Omega)$, and let $u\in C(\overline{\Omega})$
be convex and satisfy
\[
\left|u(x)-u(y)\right|\le C_{H}\left|x{-}y\right|^{\alpha}
\]
for all $x,y\in\overline{\Omega}$ and some $a\in(0,1]$. Further,
let $p\in[0,\infty)$, $\beta\in\mathbb{R}$\textup{ with $\left(1{-}\alpha\right)p-\beta=:q<1$}.
Then

\[
\int_{\Omega}\left|\nabla u(x)\right|^{p}\dist(x,\partial\Omega)^{\beta}dx\leq\mathcal{H}^{n-1}\left(\partial B_{1}^{n}\right)R^{n-1}C_{H}^{p}\frac{r^{1-q}}{1{-}q}.
\]
Especially, $u\in W^{1,p}(\Omega)$ for all $p<1{/}\left(1{-}\alpha\right)$.
\end{lem}

\begin{proof}
As a locally Lipschitz function, $u$ is differentiable a.e.\ according
to Ra\-de\-ma\-cher's theorem, and $\nabla u(x)$ coincides with
the classical derivative of $u$ in $x$ a.e. The following holds
for these points and the other null set can be ignored when integrating.
Lemma \ref{lem:gradient differenzenquotient} and the Hölder continuity
of $u$ show that
\[
\left|\nabla u(x)\right|\le\frac{\omega_{u}(\dist(x,\partial\Omega))}{\dist(x,\partial\Omega)}\le C_{H}\dist(x,\partial\Omega)^{\alpha-1},
\]
and thus
\begin{align*}
\left|\nabla u(x)\right|^{p}\dist(x,\partial\Omega)^{\beta} & \leq C_{H}^{p}\dist(x,\partial\Omega)^{-q}.
\end{align*}
Using the measure on $[0,\infty)$ defined by
\[
\sigma([a,b)):=\left|\Omega_{a}\setminus\Omega_{b}\right|,
\]
the integral of interest is bounded by
\[
\int_{\Omega}\left|\nabla u(x)\right|^{p}\dist(x,\partial\Omega)^{\beta}dx\leq C_{H}^{p}\int_{\Omega}\dist(x,\partial\Omega)^{-q}dx=C_{H}^{p}\int_{t=0}^{r}t^{-q}d\sigma.
\]
By means of Corollary \ref{cor:Hausdorff}, this can be further estimated
by 
\begin{align*}
C_{H}^{p}\int_{t=0}^{r}t^{-q}d\sigma & \leq\mathcal{H}^{n-1}(\partial\Omega)C_{H}^{p}\int_{0}^{r}t^{-q}dt\leq\mathcal{H}^{n-1}\left(\partial B_{1}^{n}\right)R^{n-1}C_{H}^{p}\frac{r^{1-q}}{1{-}q}
\end{align*}
 since $q<1$.
\end{proof}
\medskip
\begin{figure}
\begin{tikzpicture} 
\fill[gray!7] (-1.5,0) rectangle (6,3.5);
\draw (-1.5,0) -- (6,0);
\draw (4.7,.3) node{$[0,\infty)\times \R^{n-1}$};
\filldraw[fill=white] (0,0) -- (3,0) -- (4,2) -- (2,3) -- (-1,1) -- cycle; 
\draw (3,2) node {$L\Omega$}; 
\filldraw[draw=gray,fill=gray!20] (.5,0) rectangle (2.5,2); 
\draw (1.5,.5) node {$K_{2,2}$}; 
\draw[very thick] (.5,0) -- node[below]{$F$} (2.5,0);  
\draw[very thick] (.5,2) -- node[above]{$D$}(2,3) -- (2.5,2.75);
\draw[dotted] (2.5,2.75)--(2.5,2);
\end{tikzpicture} \caption{\label{fig:K in Omega}$K_{2,2}$ in $L\Omega$}
\end{figure}
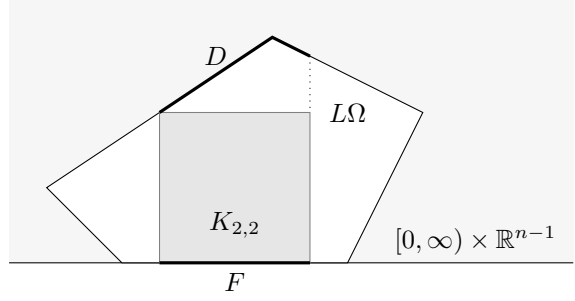
The following last preliminary lemma prepares the Theorem \ref{thm:converse_Alexandrov}
about converse bounds.
\begin{lem}
\label{lem:Randdaten annullieren}Recall that $K_{h,\rho}=(0,h)\times B_{\rho}^{n-1}.$
Assume that $\Omega$ is contained in the upper half space $(0,\infty)\times\mathbb{R}^{n-1}$,
that the boundary of $\Omega$ contains the flat hyperdisc 
\[
\tilde{F}:=\{0\}\times B_{\rho}^{n-1}\subset\partial\Omega
\]
for some $\rho>0$ and let $u\in C(\overline{\Omega})$ be convex
and $u|_{\tilde{F}}$ be affine. Then there is a map $L\in\Aff(\mathbb{R}^{n})$
such that (the configuration is depicted in Figure \ref{fig:K in Omega})
\[
F:=\{0\}\times B_{2}^{n-1}\subset L\tilde{F}\cup\overline{K}_{2,2}\subset L\overline{\Omega}\subset[0,\infty)\times\mathbb{R}^{n-1}.
\]
Further, there is an affine function $l_{g}:\mathbb{R}^{n}\rightarrow\mathbb{R}$
which depends only on $L$ and the boundary data $g:=u|_{\partial\Omega}$
such that 
\begin{align*}
u\circ L^{-1}-l_{g} & \leq0\quad\text{in }K_{2,2},\\
u\circ L^{-1}-l_{g} & =0\quad\text{on }F.
\end{align*}
\end{lem}

\begin{proof}
Since $\Omega$ is open and convex, there is a small cylinder $K$
with a base included in $\tilde{F}$. Let $L$ be a map which maps
$K$ to $K_{2,2}$. At first, subtract the affine function $u|_{\tilde{F}}\circ L^{-1}$:
$v(x):=u(L^{-1}x)-u|_{\tilde{F}}\left(0,(L^{-1}x)'\right).$ Then
$v|_{F}=0$. Let $D$ be the ``uppermost boundary points above $F$'',
\[
D:=\left\{ \left(\max\left\{ x_{1}~\middle|~(x_{1},x')\in L\overline{\Omega}\right\} ,x'\right)~\middle|~x'\in B_{2}^{n-1}\right\} .
\]
Since $K_{2,2}\subset\Omega$, the first component $x_{1}$ of a point
$x\in D$ is at least $2$. Since $\overline{D}\subset L\overline{\Omega}$
is compact, $v|_{\overline{D}}$ is bounded: $v|_{\overline{D}}\le M\in\mathbb{R}$.
Then 
\[
u_{0}(x):=v(x)-\frac{M}{2}x_{1}
\]
vanishes on $F$ and is, because of its convexity, nonpositive everywhere
between $F$ and $D$, as well as on $D$. Thus, let $l_{g}$ be the
sum of the two subtracted functions.
\end{proof}

\section{Explicit functions on cylinders with Hessian determinants bounded
from below or above\label{sec:certain explicit function}}
\begin{defn}
\label{def:main}Let $v_{h}\in C(\overline{K}_{h,1})$ be the convex
solution of 
\[
\begin{cases}
\mu_{v}=1dx & \text{in }K_{h,1}\\
v=0 & \text{on }\partial K_{h,1}.
\end{cases}
\]
For $n=2$ and $0<\varepsilon\le1/2$, let $a_{\varepsilon}(x_{1}):=x_{1}^{1-\varepsilon}$.
For $n\ge3$ let $a(x_{1}):=x_{1}^{2/n}$ and, for brevity, $a_{\varepsilon}:=a$,
so we can summarize all these cases by
\[
a_{\varepsilon}(x_{1}):=x_{1}^{2/n-\delta_{n2}\varepsilon}.
\]
For $n=2$ let
\[
\underline{a}(x_{1}):=x_{1}\left(1{-}\ln x_{1}\right)
\]
and 
\[
\overline{a}(x_{1}):=x_{1}\left(\frac{1}{2}-\ln x_{1}\right)^{1/2},
\]
both continuously extended by $\underline{a}(0):=\overline{a}(0):=0$.
Again for brevity, we set for $n\ge3$
\[
\underline{a}:=\overline{a}:=a.
\]
Further, let 
\[
b(x'):=\frac{1}{2}\left|x'\right|^{2}-1
\]
and 
\begin{align*}
w_{\varepsilon}:\overline{K}_{1,\sqrt{2}} & \rightarrow(-\infty,0]\\
(x_{1},x') & \mapsto a_{\varepsilon}(x_{1})b(x'),\\
\overline{w}:\overline{K}_{1,\sqrt{2}} & \rightarrow(-\infty,0]\\
(x_{1},x') & \mapsto\overline{a}(x_{1})b(x').
\end{align*}
\end{defn}

\begin{rem}
The functions $w_{\varepsilon}$, $w$, and $\overline{w}$ vanish
on the bottom base of the cylinder $K_{h,\sqrt{2}}$, where $x_{1}=0$,
and on the side, where $\left|x'\right|=\sqrt{2}$.
\end{rem}

\begin{lem}
\label{lem:bounds w_eps} For every $0<\varepsilon\le1/2$, there
exists a small positive radius $\rho>0$ such that $w_{\varepsilon}$
restricted to $\overline{K}_{1,\rho}$ is convex and its Hessian determinant
is bounded from below by a positive number $\lambda$. For $n=2$,
valid values are 
\[
\lambda:=\lambda_{\varepsilon}:=\frac{\varepsilon}{4}\text{ for }\rho:=\rho_{\varepsilon}:=\sqrt{\frac{\varepsilon}{2}}.
\]
\end{lem}

\begin{proof}
In $\overline{K}_{1,\rho}$, we have $b(x')\le\rho^{2}/2-1$, $\partial_{i}b(x')=x_{i}$
and $\partial_{ij}b(x')=\delta_{ij}$ for $i,j\in\{2,\dots,n\}$,
so
\[
D^{2}w_{\varepsilon}(x_{1},x')=\left(\begin{array}{cccc}
a_{\varepsilon}''(x_{1})b(x') & a_{\varepsilon}'(x_{1})x_{2} & \cdots & a_{\varepsilon}'(x_{1})x_{n}\\
a_{\varepsilon}'(x_{1})x_{2} & a_{\varepsilon}(x_{1}) & 0 & 0\\
\vdots & 0 & \ddots & 0\\
a_{\varepsilon}'(x_{1})x_{n} & 0 & 0 & a_{\varepsilon}(x_{1})
\end{array}\right).
\]
Sylvester's criterion means that $w_{\varepsilon}$ is convex, if
all the principal minors
\[
M_{k,\dots,n}:=\det\left(\begin{array}{ccc}
\partial_{kk}w_{\varepsilon} & \cdots & \partial_{kn}w_{\varepsilon}\\
\vdots & \ddots & \vdots\\
\partial_{nk}w_{\varepsilon} & \cdots & \partial_{nn}w_{\varepsilon}
\end{array}\right)
\]
are positive for all $x$ in the interior $K_{1,\rho}$. For $k\ge2$
and $\rho<\sqrt{2}$, obviously $M_{k,\dots,n}=a_{\varepsilon}(x_{1})^{n-k+1}>0$.
For $k=1$, we have to calculate the determinant of the whole matrix,
e.g.\ by the Leibniz formula, and to bound it:
\begin{align}
\det D^{2}w_{\varepsilon}(x_{1},x') & =a_{\varepsilon}''(x_{1})b(x')a_{\varepsilon}(x_{1})^{n-1}-\sum_{i=2}^{n}a_{\varepsilon}'(x_{1})^{2}x_{i}^{2}a_{\varepsilon}(x_{1})^{n-2}\nonumber \\
 & =\left(\frac{a_{\varepsilon}''(x_{1})}{a_{\varepsilon}(x_{1})}b(x')-\left(\frac{a'_{\varepsilon}(x_{1})}{a_{\varepsilon}(x_{1})}\left|x'\right|\right)^{2}\right)a_{\varepsilon}(x_{1})^{n}.\label{eq:det}
\end{align}
Note that 
\begin{align*}
\frac{a_{\varepsilon}'}{a_{\varepsilon}}(x_{1}) & =\left(\frac{2}{n}-\delta_{n2}\varepsilon\right)x_{1}^{-1},\\
\frac{a_{\varepsilon}''}{a_{\varepsilon}}(x_{1}) & =\left(\frac{2}{n}-\delta_{n2}\varepsilon\right)\left(\frac{2}{n}-1-\delta_{n2}\varepsilon\right)x_{1}^{-2},
\end{align*}
which yields
\begin{align*}
\det D^{2}w_{\varepsilon}(x_{1},x') & \geq\left(\frac{2}{n}{-}\delta_{n2}\varepsilon\right)\left(\left(1{-}\frac{2}{n}{+}\delta_{n2}\varepsilon\right)\left(1{-}\frac{1}{2}\rho^{2}\right)-\left(\frac{2}{n}{-}\delta_{n2}\varepsilon\right)\rho^{2}\right)x_{1}^{-2\delta_{n2}\varepsilon}
\end{align*}
 which is bounded away from 0 for $\varepsilon\leq1/2$ and sufficiently
small $\rho$. For $n=2$ this is
\begin{align*}
\det D^{2}w_{\varepsilon}(x_{1},x') & \geq\left(1{-}\varepsilon\right)\left(\varepsilon-\left(1{-}\frac{\varepsilon}{2}\right)\rho^{2}\right)x_{1}^{-2\varepsilon}.
\end{align*}
Since $\varepsilon\le1/2$, this is bounded by 
\[
\det D^{2}w_{\varepsilon}(x_{1},x')\geq\frac{\left(1{-}\varepsilon\right)\varepsilon}{2}\geq\frac{\varepsilon}{4}=\lambda_{\varepsilon},
\]
if
\begin{align*}
\rho_{\varepsilon} & :=\sqrt{\frac{\varepsilon}{2}}\leq\sqrt{\frac{\varepsilon}{2{-}\varepsilon}}.
\end{align*}
\end{proof}
\begin{lem}[Strengthened AMP on a square]
\label{lem:AMP on a square}There exists a dimensional constant $C_{n}>0$
such that the solution $v_{2}:\overline{K}_{2,1}\to(-\infty,0]$ and
the function $\underline{a}$ from Definition \ref{def:main} satisfy
for every $(x_{1},x')\in\overline{K}_{1,1}$
\[
\left|v_{2}(x_{1},x')\right|\le C_{n}\underline{a}(x_{1}).
\]
\end{lem}

\begin{proof}
At first, we prove it only for $v_{1}$ on $\overline{K}_{1,1}$.
Take the function $w_{\varepsilon}$, the lower bound $\lambda\leq\det D^{2}w_{\varepsilon}$
and the radius $\rho$ from Lemma \ref{lem:bounds w_eps}, stretch
the domain by $1/\rho$ in the radial $x'$ directions, and multiply
the function by $\lambda^{-1/n}\rho^{-2(n{-}1)/n}$:
\begin{align*}
w_{1,\varepsilon}:\overline{K}_{1,1} & \rightarrow(-\infty,0)\\
w_{1,\varepsilon}(x_{1},x') & :=\lambda^{-1/n}\rho^{-2(n{-}1)/n}w_{\varepsilon}(x_{1},\rho x').
\end{align*}
The factor is chosen such that according to the reduction to a normalized
problem (Lemma \ref{lem:Reduction to normalized}),
\[
\det D^{2}w_{1,\varepsilon}\ge1.
\]
The function $w_{1,\varepsilon}\leq0$ on $\partial K_{1,1}$. Therefore,
the comparision principle yields $0\ge v\ge w_{1,\varepsilon}$. For
$n\ge3$, note that $b\leq1$ and we are done. For $n=2$, we continue
as follows:
\begin{align*}
\left|v(x)\right| & \leq\left|w_{1,\varepsilon}(x)\right|\leq\lambda_{\varepsilon}^{-1/2}\rho_{\varepsilon}^{-1}x_{1}^{1-\varepsilon}=\sqrt{8}x_{1}^{1-\varepsilon}/\varepsilon.
\end{align*}
In order to minimize this for fixed $x_{1}$ with respect to $\varepsilon$,
we differentiate:
\[
\frac{\partial}{\partial\varepsilon}\frac{x_{1}^{1-\varepsilon}}{\varepsilon}=\frac{x_{1}^{1-\varepsilon}}{\varepsilon}\left(-\ln x_{1}-\frac{1}{\varepsilon}\right).
\]
Thus, the minimum is expected at $\varepsilon=-1/\ln x_{1}$. Therefore
\begin{align*}
\left|v(x_{1},x_{2})\right| & \leq\sqrt{8}x_{1}^{1+1/\ln x_{1}}\left(-\ln x_{1}\right)=\sqrt{8}ex_{1}\left(-\ln x_{1}\right),
\end{align*}
if $\varepsilon=-1/\ln x_{1}\le1/2$, i.e.\ for $x_{1}\le e^{-2}$.
Since for all remaining $x_{1}\in[e^{-2},1]$,
\[
\left|v(x)\right|\le\left|w_{1,1/2}(x)\right|\leq\sqrt{8}x_{1}^{1-1/2}/(1/2)\le C_{2}x_{1}\left(1-\ln x_{1}\right),
\]
for a sufficiently large $C_{2}\geq\sqrt{8}e$, altogether we get
the global bound
\[
\left|v(x)\right|\le C_{2}x_{1}\left(1-\ln x_{1}\right)=C_{2}\underline{a}(x_{1}).
\]
 To extend this estimate to $v_{2}$ on $K_{2,1}$, note that on each
line $[0,1]\times\{x'\}$, $w_{\varepsilon}$ attains its minimum
at $x_{1}=1$, so it can be convexly extended by reflection:
\[
w_{\varepsilon}(x_{1},x'):=w_{\varepsilon}(2{-}x_{1},x')\quad\text{for }1<x_{1}\le2,
\]
preserving $\mu_{w_{\varepsilon}}\ge\lambda_{\varepsilon}$. Thus,
the estimates can be extended to the doubled domain.
\end{proof}
\begin{rem}
Sharper estimates can be achieved, if $a_{\varepsilon}$ and $b_{\varepsilon}(x'):=b_{\varepsilon}(\left|x'\right|)$
are defined as the solutions of the differential equations 
\begin{align*}
a_{\varepsilon}''a_{\varepsilon}^{n-1} & =-x^{-2\delta_{n2}\varepsilon}\\
a_{\varepsilon}(0)=a_{\varepsilon}(2) & =0\\
\left(\frac{2}{n}{-}\delta_{n2}\varepsilon\right)\left(1{-}\frac{2}{n}{+}\delta_{n2}\varepsilon\right)b_{\varepsilon}''^{n-1}b_{\varepsilon}+\left(\frac{2}{n}{-}\delta_{n2}\varepsilon\right)^{2}b_{\varepsilon}''^{n{-}2}b_{\varepsilon}'^{2} & =-1\\
b_{\varepsilon}(-1)=b_{\varepsilon}(1) & =0.
\end{align*}
This improves the constant $C_{n}$, but not the quality of the estimate.
\end{rem}

\begin{lem}[Converse AMP on a rectangle]
\label{lem:upper bound}The Monge--Ampère measure $\mu_{\overline{w}_{*}}$
of the convex envelope of $\overline{w}(x)=\overline{a}(x_{1})b(x')$
from Definition \ref{def:main} has a density which is bounded from
\emph{above}.
\end{lem}

\begin{proof}
According to \cite[Proposition A.35, p.~184]{Figalli2017}, the Monge--Ampère
measure of the convex envelope can be bounded by
\[
\mu_{\overline{w}_{*}}(E)\leq\int_{E}\chi_{\left\{ x~\middle|~\overline{w}(x)=\overline{w}_{*}(x)\right\} }\det D^{2}\overline{w}\,dx
\]
 for all Borel sets $E$, so it suffices to bound $\det D^{2}\overline{w}$
from above. Start with (\ref{eq:det}) for $\overline{w}$ instead
of $w_{\varepsilon}$ and recall that $\overline{a}>0>\overline{a}''$,
$b\ge-1$ to estimate
\begin{align*}
\det D^{2}\overline{w}(x_{1},x') & =\left(\frac{\overline{a}''(x_{1})}{\overline{a}(x_{1})}b(x')-\left(\frac{\overline{a}'(x_{1})}{\overline{a}(x_{1})}\left|x'\right|\right)^{2}\right)\overline{a}(x_{1})^{n}\\
 & \le-\overline{a}''(x_{1})\overline{a}(x_{1})^{n-1}\\
 & =\begin{cases}
\frac{1-\ln x_{1}}{1-2\ln x_{1}}\le1 & \text{for }n=2,\\
\frac{2}{n}\left(1{-}\frac{2}{n}\right) & \text{otherwise.}
\end{cases}
\end{align*}
\end{proof}

\section{\label{sec:results}Improved Hölder continuity and $W^{1,p}$ estimates
for functions with bounded Hessian determinant.}
\begin{lem}[Strengthened Alexandrov Maximum Principle]
\label{lem:new_Alexandrov}Let $L\in\Aff(\mathbb{R}^{n})$ with $L\Omega\subset B_{1}^{n}$,
$u\in C(\overline{\Omega})$ be a convex function with Monge--Ampère
measure $\mu_{u}\le\Lambda dx$ and $u|_{\partial\Omega}=0$. For
the same constant $C_{n}>0$ as in Lemma \ref{lem:AMP on a square}
and $\underline{a}$ from Definition \ref{def:main} it holds that
\begin{align*}
\left|u(x)\right| & \le C_{n}\left|\det L\right|^{-2/n}\Lambda^{1/n}\underline{a}(\left|L\right|\dist(x,\partial\Omega))
\end{align*}
for $\underline{a}$ as in Definition \ref{def:main} and for all
$x\in\Omega$.
\end{lem}

\begin{rem}
For $n=2$ and any $\alpha<1$, $\left|u(x)\right|\lesssim\underline{a}(\left|L\right|\dist(x,\partial\Omega))$
implies $\left|u(x)\right|\lesssim\dist(x,\partial\Omega)^{\alpha}$.
\end{rem}

\begin{proof}
Lemma \ref{lem:Reduction to normalized} reduces the statement to
the normalized case where $L=\id$, $\Lambda=1$. Let $v_{2}$ be
as in Definition \ref{def:main}. The following configuration is depicted
in Figure \ref{fig:Omega in K}. Let $x_{0}$ be a proximum of $x$
on $\partial\Omega$ and we can assume that after a motion, $\Omega$
is included in $K_{2,1}$ and touches the bottom base $B_{B}:=\{0\}\times\overline{B}_{1}^{n-1}$
of $K_{2,1}$ at $x_{0}$.
\begin{figure}
\begin{tikzpicture}[scale=.8] 
\draw (0,0) -- (3,0) -- (4,4) -- (2,4) -- (0,1) -- cycle; 
\draw (2.5,1.5) node {$\Omega$}; 
\draw[gray] (-1,0) rectangle +(5,5); 
\draw (-.5,4.5) node {$\tilde K$}; 
\draw[very thick] (-1,0) -- node[below left]{$B_B$} (4,0);  
\draw (2,1) circle [radius=.3mm] node[above]{$x$};  \draw[dotted] (2,1) -- (2,0);  \draw (2,0) circle [radius=.3mm] node[below]{$x_0$}; \end{tikzpicture} 

\caption{\label{fig:Omega in K}$\Omega$ in $K_{2,1}$}
\end{figure}
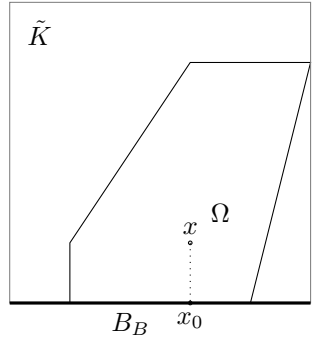
 The comparison principle yields that $\left|u\right|\leq\left|v\right|$.
Since $\Omega$ is included in a ball of radius 1, it holds that $x_{1}=\left|x{-}x_{0}\right|=\dist(x,\partial\Omega)\le1$.
Thus, Lemma \ref{lem:AMP on a square} implies
\[
\left|u(x)\right|\le\left|v(x)\right|\le C_{n}\underline{a}\left(x_{1}\right)=C_{n}\underline{a}\left(\dist(x,\partial\Omega)\right).
\]
\end{proof}
With the equivalence of the AMP and Hölder continuity (\ref{eq:H=0000F6lderstetigkeit AMP})
and the subadditivity between the Monge--Ampère measure and the modulus
of continuity (Lemma \ref{lem:subadditivity MAM modul}), this implies
our main result:
\begin{thm}[Hölder continuity]
\label{thm:H=0000F6lderstetigkeit}Let $\Omega\subset\mathbb{R}^{n}$
be open, bounded and convex, $L\in\Aff(\mathbb{R}^{n})$ which maps
$\Omega$ into $B_{1}^{n}$, $f:\Omega\to\mathbb{R}$ with $f\le\Lambda<\infty$,
and $g$ be a convex function on $\overline{\Omega}$. Let $u$ be
the continuous convex solution of (\ref{eq:Alexandrov formulation}).
For the same dimensional constant $C_{n}>0$ as in Lemma \ref{lem:AMP on a square},
it holds for the modulus of continuity of $u$ that
\[
\omega_{u}(\delta)\leq\omega_{g}(\delta)+C_{n}\left|\det L\right|^{-2/n}\Lambda^{1/n}\begin{cases}
\left|L\right|\delta\left(1-\ln(\left|L\right|\delta)\right) & \text{for }n=2,\\
\left(\left|L\right|\delta\right)^{2/n} & \text{otherwise}.
\end{cases}
\]
Particularly, if $g\in C^{\alpha}(\overline{\Omega})$ for some $\alpha<1$,
then $u\in C^{\min\{\alpha,2/n\}}(\overline{\Omega})$. 
\end{thm}

\begin{proof}
Apply (\ref{eq:H=0000F6lderstetigkeit AMP}) and Lemma \ref{lem:subadditivity MAM modul},
where $u_{1}$ is the solution for zero boundary conditions from the
last lemma, $u_{2}=g$, and $u_{12}$ is the seeked solution $u$.
\end{proof}
Together with Lemma \ref{lem:H=0000F6lder Sobolev}, this implies
the following immediately.
\begin{cor}[$\dist$-weighted $W^{1,p}$ regularity]
\label{thm:w1p}Let $\Omega\subset B_{R}^{n}$, and $f$ and $u$
be as in Theorem 18, where $g\in C^{\alpha}(\overline{\Omega})$ is
convex. Then the following holds true for the weak derivative $\nabla u$
of $u$. Let $p,\beta\in[0,\infty)$\textup{ with $\left(1-\min\{2/n,\alpha\}\right)p-\beta=:q<1$}.
Then
\begin{align*}
\int_{\Omega}\left|\nabla u(x)\right|^{p}\dist(x,\partial\Omega)^{\beta}dx & \le C(n,R,\Lambda,p,\beta,\omega_{g})<\infty.
\end{align*}
In particular, $u\in W^{1,p}(\Omega)$ for all $p<\left(1-\min\{2/n,\alpha\}\right)^{-1}$.
\end{cor}

\begin{rem}
(The dependence of $\left\Vert \nabla u\right\Vert _{p}$ on a normalizing
map and $\Lambda$.) Let $g:=0$ and let $L\in\Aff(\mathbb{R}^{n})$
such that
\[
B_{1}^{n}\subset L\Omega\subset B_{n}^{n}.
\]
Then there is a constant $C(n,p)>0$ such that
\begin{equation}
\left(\frac{\int_{\Omega}\left|\nabla u(x)\right|^{p}dx}{\mu(\Omega)}\right)^{1/p}\le\frac{\left|L\right|}{\left(\det L\right)^{2/n}}\Lambda^{1/n}C(n,p).\label{eq:dependence on L and Lambda}
\end{equation}
\end{rem}

\begin{proof}
Write $u$ as $\Lambda^{1/n}\det L^{-2/n}v\circ L$, such that $v$
maps from the normalized domain $L\Omega\subset B_{n}^{n}$ and has
the Monge--Ampère measure $\det D^{2}v(Lx)=\det D^{2}u(x)/\Lambda\leq1$.
Corollary \ref{thm:w1p} bounds the left-hand side of (\ref{eq:dependence on L and Lambda})
for $v$ instead of $u$ dependently only on $n$ and $p$. Meanwhile,
Lemma \ref{lem:Reduction to normalized} says that $\left|\nabla u(x)\right|\le\det L^{-2/n}\left|L\right|\left|\nabla v(Lx)\right|\Lambda^{1/n}$,
so the same factor arises if we estimate the $p$-mean of the pointwise
norms of the gradients: 
\[
\left(\frac{\int_{\Omega}\left|\nabla u(x)\right|^{p}dx}{\mu(\Omega)}\right)^{1/p}\le\frac{|L|}{\det L^{2/n}}\left(\frac{\int_{L\Omega}\left|\nabla v(x)\right|^{p}dx}{\mu(L\Omega)}\right)^{1/p}\le\frac{|L|}{\det L^{2/n}}\Lambda^{1/n}C(n,p).
\]
\end{proof}

\section{\label{sec:converse bounds}Converse bounds for the exponents}

Now we show that the exponent of $\delta$ in Theorem \ref{thm:H=0000F6lderstetigkeit}
cannot be larger than $2/n$ and that $\int_{\Omega}\left|\nabla u(x)\right|^{n/(n-2)}dx=\infty$,
if the boundary of $\Omega$ contains a flat piece where $u$ is affine
and $\det D^{2}u\ge\lambda>0$.
\begin{thm}[Converse Hölder and Sobolev bounds]
\label{thm:converse_Alexandrov}Assume that the boundary of $\Omega$
contains some flat boundary piece, namely that
\[
\tilde{F}=\{0\}\times B_{\rho}^{n-1}\subset\overline{\Omega}\subset[0,\infty)\times\mathbb{R}^{n-1}.
\]
Let $u\in C(\overline{\Omega})$ be a convex function with Monge--Ampère
measure $\mu_{u}\ge\lambda dx>0$ and let $u|_{\tilde{F}}$ be affine.
Let $L\in\Aff(\mathbb{R}^{n})$ be the map from Lemma (\ref{lem:Randdaten annullieren})
such that
\[
F:=\{0\}\times B_{2}^{n-1}\subset L\tilde{F}\cup\overline{K}_{2,2}\subset L\overline{\Omega}\subset[0,\infty)\times\mathbb{R}^{n-1}.
\]
\begin{enumerate}
\item Then there exists a constant $C_{L,g}$, depending only on $L$ and
$u|_{\partial\Omega}$, and a constant $C_{n}>0$, depending only
on $n$, such that for every $x\in L^{-1}K_{1,1}$ and its projection
$x_{0}:=L^{-1}\left(0,\left(Lx\right)'\right)$ it holds that
\begin{align*}
 & u\left(x_{0}\right)-u(x)\\
 & \geq-C_{L,g}\left|x{-}x_{0}\right|+\frac{C_{n}\lambda^{1/n}}{\left|\det L\right|^{2/n}}\begin{cases}
\left|L\right|\left|x{-}x_{0}\right|\left(1/2-\ln\left(\left|L\right|\left|x{-}x_{0}\right|\right)\right)^{1/2} & \text{if }n=2,\\
\left(\left|L\right|\left|x{-}x_{0}\right|\right)^{2/n} & \text{otherwise.}
\end{cases}
\end{align*}
In particular, $u$ is not Lipschitz continuous and not Hölder continuous
with exponent $\alpha>2/n$.
\item For $n\ge2$, $\left\Vert \nabla u\right\Vert _{\infty}=\infty$.
For $n\ge3$, $\left\Vert \nabla u\right\Vert _{n/\left(n{-}2\right)}=\infty$.
\end{enumerate}
\end{thm}

\begin{proof}
(1) Lemma \ref{lem:Reduction to normalized} reduces the statement
to the case where $L=\id$. Let $l_{g}$ be the affine map given by
Lemma \ref{lem:Randdaten annullieren}. The function $u_{0}:=u-l_{g}$
is non-positive on $\partial K_{2,2}$ and vanishes on $F$. Thus,
for $x\in K_{2,2}$ we have
\[
u(0,x')-u(x)=-u_{0}(x)+l_{g}(0,x')-l_{g}(x).
\]
Hence, letting $C_{L,g}$ be the Lipschitz constant of $l_{g}$, it
remains to find a constant $C_{n}>0$ such that for every $x\in K_{1,1}$
\begin{equation}
\left|u_{0}(x)\right|\ge\lambda^{1/n}C_{n}\overline{a}(x_{1})\label{eq:converse_Alexandrov}
\end{equation}
for $\overline{a}$ from Definition \ref{def:main}. The scaling property
(\ref{eq:scaling properties}) reduces the statement to the normalized
case where $\lambda=1$. Let $B_{T}:=\{1\}\times\overline{B}_{\sqrt{2}}^{n-1}$
be the top base of $K_{1,\sqrt{2}}$. If $u_{0}$ vanished anywhere
in $\Omega$, the convexity would constrain it to vanish at all, contradicting
$\mu_{u_{0}}\geq dx$. So $u_{0}<0$ in $\Omega$ and attains a negative
maximum on the compactum $B_{T}\subset K_{2,2}$. Therefore, there
exists a small constant $C>0$ such that for the convex envelope $\overline{w}_{*}$
of $\overline{w}(x)=\overline{a}(x_{1})b(x')$ from Definition \ref{def:main},
it holds that
\begin{itemize}
\item $C\overline{w}_{*}\ge u_{0}$ on $B_{T}$ and anyhow $C\overline{w}_{*}=0\ge u_{0}$
everywhere else on $\partial K_{1,\sqrt{2}}$, 
\item $\mu_{C\overline{w}_{*}}\le1dx\leq\mu_{u_{0}}$.
\end{itemize}
(The latter condition is enabled by Lemma \ref{lem:upper bound} and
the scaling properties.) Then the comparision principle (\ref{eq:comparision principle})
says that $C\overline{w}_{*}\ge u_{0}|_{K_{1,\sqrt{2}}}$ and we get
for $x\in K_{1,1}$, where $-b(x')=1-x^{2}/2\geq1/2$, 
\begin{align*}
-u_{0}(x) & \ge-C\overline{w}_{*}(x)\geq-C\overline{w}(x)\geq\frac{1}{2}C\overline{a}(x_{1}).
\end{align*}

(2) The statements are independent of addition or subtraction of an
affine function, so it suffices to show it for $u_{0}$. For $n=2$
we have just seen that $u_{0}$ is not Lipschitz at the flat piece.
As a continuous function on the closure of an open domain, it is not
possible to make it Lipschitz by removing a null set from the domain,
so $\left\Vert \nabla u_{0}\right\Vert _{\infty}=\infty$. For $n\ge3$,
assume that the affine map $L$ from the first step is the identity
without loss of generality. Start with the estimate
\[
\int_{\Omega}\left|\nabla u_{0}(x)\right|^{p}dx\geq\int_{B_{1}^{n-1}}\int_{0}^{1}\left|\frac{\partial u_{0}}{\partial x_{1}}(x_{1},x')\right|^{p}dx_{1}\,dx'.
\]
We claim that $\int_{0}^{1}\left|\partial_{1}u_{0}(x_{1},x')\right|^{p}dx_{1}=\infty$
for each $x'\in B_{1}^{n-1}$, such that the whole integral is infinite.
If it were not infinite, for every $\delta>0$ there would have to
be an $\varepsilon>0$ such that 
\[
\int_{0}^{\varepsilon}\left|\partial_{1}u_{0}(x_{1},x')\right|^{p}dx_{1}<\delta
\]
by the monotone convergence theorem. Contradicting that, the Jensen
inequality, the fundamental theorem of calculus (which is applicable
because a continuous convex function on a compactum is absolutely
continuous), and (\ref{eq:converse_Alexandrov}) yield
\begin{align*}
\int_{0}^{\varepsilon}\left|\frac{\partial u_{0}}{\partial x_{1}}(x_{1},x')\right|^{p}dx_{1} & \ge\varepsilon\left(\frac{u_{0}(\varepsilon,x')-u_{0}(0,x')}{\varepsilon}\right)^{p}\\
 & \ge C\varepsilon^{1+p(2/n{-}1)}=C
\end{align*}
for $p=n/\left(n{-}2\right)$.
\end{proof}

\section{Conclusion and a remaining question}

The solution $u$ of a problem with right-hand side 1 and affine boundary
data on a flat boundary piece grows like $x^{2/n}$ near the flat
part of the boundary for $n\ge3$. For $n=2$, it has been established
that the modulus of continuity of $u$ lies between
\[
\delta(1{-}\log\delta)^{1/2}\lesssim\omega_{u}(\delta)\lesssim\delta(1{-}\log\delta).
\]
But what is an explicit function $a$ with
\[
a\lesssim\omega_{u}\lesssim a?
\]

\bibliographystyle{alpha}
\bibliography{literatur}

\begin{thebibliography}{Sav13b}

\bibitem[Ale58]{alexandrov}
Aleksandr~D.\ Aleksandrov.
\newblock Dirichlet's problem for the equation {{\(\text{Det}\| z_{ij}\|
  =\varphi (z_ 1,\ldots ,z_ n,z,x_ 1,\ldots ,x_ n)\)}}. {I}.
\newblock {\em Vestn. Leningr. Univ., Mat. Mekh. Astron.}, 13(1):5--24, 1958.

\bibitem[Br{\'e}83]{brezis}
Ha{\"{\i}}m Br{\'e}zis.
\newblock {\em Analyse fonctionnelle: th{\'e}orie et applications}.
\newblock Collection Math{\'e}matiques appliqu{\'e}es pour la ma{\^\i}trise.
  Paris: Masson, 1983.

\bibitem[Caf90]{caffarelli_localization}
Luis~A.\ Caffarelli.
\newblock A localization property of viscosity solutions to the
  {Monge}--{Amp{\`e}re} equation and their strict convexity.
\newblock {\em Ann. of Math. (2)}, 131(1):129--134, 1990.

\bibitem[EG15]{Evans}
Lawrence~Craig Evans and Ronald~F. Gariepy.
\newblock {\em Measure theory and fine properties of functions}.
\newblock Textb. Math. Boca Raton, FL: CRC Press, 2nd revised edition, 2015.

\bibitem[Fig17]{Figalli2017}
Alessio Figalli.
\newblock {\em The {M}onge--{A}mp\`ere equation and its applications}.
\newblock Zurich Lectures in Advanced Mathematics. European Mathematical
  Society (EMS), Z\"{u}rich, 2017.

\bibitem[LN14]{le_w2p}
Nam~Q.\ Le and Truyen Nguyen.
\newblock Global {$W^{2,p}$} estimates for solutions to the linearized
  {Monge}--{Amp\`ere} equations.
\newblock {\em Math. Ann.}, 358, 04 2014.

\bibitem[LN17]{le_global_w1p}
Nam~Q. Le and Truyen Nguyen.
\newblock Global {{\(W^{1,p}\)}} estimates for solutions to the linearized
  {Monge}-{Amp{\`e}re} equations.
\newblock {\em J. Geom. Anal.}, 27(3):1751--1788, 2017.

\bibitem[Sav12]{savin_localization}
Ovidiu Savin.
\newblock A localization property at the boundary for {Monge}-{Amp{\`e}re}
  equation.
\newblock In {\em Advances in geometric analysis. Collected papers of the
  workshop on geometry in honour of Shing-Tung Yau's 60th birthday, Warsaw,
  Poland, April 6--8, 2009}, pages 45--68. Somerville, MA: International Press;
  Beijing: Higher Education Press, 2012.

\bibitem[Sav13a]{savin_w2p}
O.~Savin.
\newblock Global {{\(W^{2,p}\)}} estimates for the {Monge}-{Amp{\`e}re}
  equation.
\newblock {\em Proc. Am. Math. Soc.}, 141(10):3573--3578, 2013.

\bibitem[Sav13b]{savin_pointwise}
Ovidiu Savin.
\newblock Pointwise {$C^{2,\alpha}$} estimates at the boundary for the
  {Monge--Ampère} equation.
\newblock {\em Journal of the American Mathematical Society}, 26:63--99, 01
  2013.

\bibitem[Ste18]{stefani}
Giorgio Stefani.
\newblock On the monotonicity of perimeter of convex bodies.
\newblock {\em J. Convex Anal.}, 25(1):93--102, 2018.

\bibitem[TU83]{trudinger_urbas}
Neil~S. Trudinger and John I.~E. Urbas.
\newblock The {Dirichlet} problem for the equation of prescribed {Gauss}
  curvature.
\newblock {\em Bull. Aust. Math. Soc.}, 28:217--231, 1983.

\bibitem[TW08]{trudinger_wang}
Neil~S. Trudinger and Xu-Jia Wang.
\newblock The {Monge}-{Amp{\`e}re} equation and its applications.
\newblock In {\em Handbook of geometric analysis. No. 1}, pages 467--524.
  Somerville, MA: International Press; Beijing: Higher Education Press, 2008.

\bibitem[Urb88]{urbas}
John I.~E. Urbas.
\newblock Global {H{\"o}lder} estimates for equations of {Monge}-{Amp{\`e}re}
  type.
\newblock {\em Invent. Math.}, 91(1):1--29, 1988.

\bibitem[Urb91]{urbas_gauss}
John I.~E. Urbas.
\newblock Boundary regularity for solutions of the equation of prescribed
  {Gauss} curvature.
\newblock {\em Ann. Inst. Henri Poincar{\'e}, Anal. Non Lin{\'e}aire},
  8(5):499--522, 1991.

\bibitem[Win09]{winter_n_w1p_at_the_b}
Niki Winter.
\newblock {$W^{2,p}$}- and {$W^{1,p}$}-estimates at the boundary for solutions
  of fully nonlinear, uniformly elliptic equations.
\newblock {\em Z. Anal. Anwend.}, 28:129--164, 01 2009.

\end{thebibliography}

\end{document}